\documentclass[10pt,a4paper,abstract,keyword]{scrartcl}%

\usepackage[utf8]{inputenc}
\usepackage[T1]{fontenc}
\usepackage[top=2cm, bottom=2cm, left=2cm, right=2cm]{geometry}
\usepackage[english]{babel}
\usepackage{authblk}
\usepackage{enumitem}
\usepackage{cite}

\usepackage{amssymb}
\usepackage{graphicx,url,amsthm}
\usepackage{siunitx}
\usepackage{pstricks}
\usepackage{pst-grad} 
\usepackage{pst-plot} 
\usepackage[babel]{csquotes}
\usepackage{amsmath}
\usepackage{array}
 \usepackage[caption=false,font=normalsize,labelfont=sf,textfont=sf]{subfig}

\usepackage{url}
\usepackage[babel]{csquotes}

\usepackage{flushend}
\theoremstyle{plain}
\newtheorem{theorem}{Theorem}

\theoremstyle{remark}
\newtheorem{example}{Example}

\newcommand{\eps}{\varepsilon}
\newcommand{\f}{\frac}
\newcommand{\pa}{\partial}

\newcommand{\Omg}{\Omega}

\newcommand{\alp}{\alpha}
\newcommand{\lap}{\Delta}
\newcommand{\xh}{\hat{\mathbf{x}}}
\newcommand{\fd}{\hat{\mathbf{d}}}
\newcommand{\fr}{\mathbf{r}}
\newcommand{\fx}{\mathbf{x}}

\newcommand{\fD}{\mathbf{D}}

\newcommand{\fz}{\mathbf{z}}
\newcommand{\bke}[1]{\left( #1 \right)}
\newcommand{\bkc}[1]{\langle #1 \rangle}

\newcommand{\norm}[1]{\left\Vert #1 \right\Vert}
\newcommand{\abs}[1]{\left| #1 \right|}

\renewcommand{\vec}[1]{\mathbf{#1}}
\newcommand{\e}{\mathrm{e}}

\DeclareMathOperator{\BesselJ}{J}
\newcommand{\ci}{\mathrm{i}}

\title{Analysis and improvement of direct sampling method in the mono-static configuration}
\author[1]{S.~Kang}
\author[1]{M.~Lambert}
\author[2]{W.-K. ~Park}
\affil[1]{GeePs $|$ Group of electrical engineering - Paris, CNRS, CentraleSupélec, Univ. Paris-Sud, Université Paris-Saclay, Sorbonne Université, 
	3 \& 11 rue Joliot-Curie, Plateau de Moulon 91192 Gif-sur-Yvette CEDEX, France}
\affil[2]{Department of Information Security, Cryptology, and Mathematics, Kookmin University, Seoul, 02707, Korea.}

\date{}

\begin{document}

%
%
%
\maketitle

\begin{abstract}
The recently introduced non-iterative imaging method entitled \enquote{direct sampling method} (DSM) is known to be fast, robust, and effective for inverse scattering problems in the multi-static configuration but fails when applied to the mono-static one. To the best of our knowledge no explanation of this failure has  been provided yet. Thanks to the framework of the asymptotic and the far-field hypothesis in the 2D scalar configuration an analytical expression of the DSM indicator function in terms of the Bessel function of order zero and sizes, shapes and permittivities of the inhomogeneities is obtained and the theoretical reason of the limitation identified.  A modified version of DSM is then proposed in order to improve the imaging method. The  theoretical results are supported by numerical results using synthetic data.
\paragraph{keywords:} Non-iterative imaging method, direct sampling method, mono-static configuration, Bessel function, numerical results
\end{abstract}

\section{Introduction}
%
%
%
%
The 2D inverse scattering problem is an important topic due to potential applications  in modern human life, e.g., biomedical imaging \cite{A1,CZBN,SOJHCK}, non-destructive evaluation \cite{NDT_review,PGBM,VS}, synthetic aperture radar (SAR) imaging \cite{cetin2014sparsity,zhang2010resolution,zhang2015ofdm,LSM_gpr}, ground penetrating radar (GPR) \cite{ConstructionAndBuildingMaterials-2016,KM_gpr,SS}. However, because of its inherent non-linearity and ill-posedness, it is difficult to solve. Among the various imaging methods, non-iterative-type algorithms are of interest due to expected numerical simplicity and low computational cost, for example, MUltiple SIgnal Classification (MUSIC), linear sampling method (LSM), topological derivative, Kirchhoff migration, direct sampling method (DSM), etc. Related works can be found in \cite{KM_crack,Ammari_music,Kisch_lsm2d,dsm2d_farfield,LR,Park} and references therein.
Even though these methods can provide good results with multi-static data, they may fail with mono-static ones the due to lack of information arising to great assumption  from  inherent limitation. However, since the mono-static configuration is encountered in various applications such as GPR, SAR,  deep understanding and development of effective algorithms is needed. 

In the present work, we focus only onto DSM in the mono-static configuration because of its wide applicability, and various advantages like (i) it only needs  a few (e.g., one or two) incident fields, and (ii) it does not need any additional operation (singular value decomposition, defining an orthogonal projection operator and solving ill-posed linear integral equations, etc.). We refer to \cite{dsm2d_ito1, dsm2d_farfield} for  details. Though, a new intuitive indicator function of DSM in the mono-static configuration has already been proposed in \cite{bektas2016direct}, no theoretical explanation has been given yet to explain the failure of the classical DSM approach in such a configuration. Recently, in \cite{InverseProblems-Kang-2017}, the authors have investigated the mathematical structure of the DSM indicator function in the multi-static configuration using near-field data, proposed an improved version and confirmed its link with the classical Kirchhoff migration technique. Following a similar path but under the far-field hypothesis the mathematical structure of the indicator function of DSM based on the asymptotic formula of the scattered fields is proposed here and the limitation of traditional DSM in the mono-static configuration is identified. 
According to our analysis, a new indicator function of the direct sampling method  is introduced and analyzed in order to improve the imaging performance of DSM in this mono-static configuration.

In Section \ref{sec:2}, the 2D direct scattering problem and its far-field pattern are presented. The traditional DSM with far-field pattern is reminded in Section \ref{sec:3}. Section \ref{sec:4} is dedicated to the mono-static configuration,  the mathematical structure of DSM being outlined and the modified DSM (MDSM) proposed. Numerical simulations illustrating our theoretical results are presented in Section \ref{sec:5}. Conclusions and perspectives follow in Section \ref{sec:6}.

\section{Two-dimensional direct scattering problem and far-field pattern}\label{sec:2}
In this section, the two-dimensional direct scattering problem is sketched in the presence of a set of small dielectric inhomogeneities (Fig.~\ref{Location1-1}). We denote $\tau_m$ a small dielectric inhomogeneity defined as $\tau_m = \fr_m + \alp_m\fD_m$, where $\fr_m$ is the location of $\tau_m$, $\fD_m$ is a simply connected domain with smooth boundary and $\alpha_m$ characterizes its size (Fig.~\ref{Location1-2}). We denote $\tau=\bigcup_{m}\tau_{m}$, $m=1,2,\cdots,M$  a collection of $\tau_m$ and $\Omega$ the region of interest (ROI) such that $\tau_m\subset\Omega$ for all $m$. We assume that $\tau_m$ are well-separated small balls with radius $\alpha_m$, i.e., there exists $d_0\in\mathbb{R}$ such that $
0< d_0 < |\fr_m-\fr_{m'}|$ for all $m\ne m'$, $m=1,2,\cdots,M$.

\begin{figure}
	\centering
	\subfloat[Scattering problem]{\label{Location1-1}\centering
		\begin{minipage}[c][4.5cm][c]{.45\linewidth}\centering
			\psscalebox{0.9 0.9} 
			{\centering\small
				\begin{pspicture}(0,-2.1918733)(4.345445,2.1918733)
				\definecolor{colour0}{rgb}{1.0,0.8,0.0}
				\rput[bl](3.9923842,0.009832429){$\hat{\vec{d}}_{1}$}
				\psframe[linecolor=black, linewidth=0.04, dimen=outer](2.9774606,0.81569976)(0.9346035,-1.2271574)
				\pscircle[linecolor=black, linewidth=0.04, fillstyle=solid,fillcolor=colour0, dimen=outer](1.6399606,0.38444978){0.21875}
				\pscircle[linecolor=black, linewidth=0.04, fillstyle=solid,fillcolor=colour0, dimen=outer](2.4962106,-0.23430023){0.2375}
				\pscircle[linecolor=black, linewidth=0.04, fillstyle=solid,fillcolor=colour0, dimen=outer](1.3249097,-0.8655502){0.21875}
				\rput[bl](1.5087106,0.27819976){$\tau_{1}$}
				\rput[bl](2.3462107,-0.35930023){$\tau_{2}$}
				\rput[bl](1.2013128,-0.9593002){$\tau_{3}$}
				\rput[bl](2.4587107,-1.1343002){$\Omega$}
				\psarc[linecolor=black, linewidth=0.01, linestyle=dashed, dash=0.17638889cm 0.10583334cm, dimen=outer](1.899272,-0.29284614){1.89375}{60.0}{320.0}
				\psline[linecolor=black, linewidth=0.04](3.5932333,-1.0465858)(3.1145358,-1.6876823)
				\psline[linecolor=black, linewidth=0.04](3.7225213,-1.1394589)(3.2438235,-1.7805554)
				\psline[linecolor=black, linewidth=0.04, arrowsize=0.05291667cm 3.96,arrowlength=1.63,arrowinset=0.0]{<-}(2.9945424,-1.001094)(3.8225176,-1.5839916)(3.8225176,-1.5839916)
				\psline[linecolor=black, linewidth=0.04, arrowsize=0.05291667cm 3.96,arrowlength=1.63,arrowinset=0.0]{<-}(3.8084974,-1.7836881)(2.9805224,-1.2007904)(2.9805224,-1.2007904)
				\psline[linecolor=black, linewidth=0.04](2.5774577,1.4218062)(3.2665553,1.0152346)
				\psline[linecolor=black, linewidth=0.04](2.6558068,1.5603781)(3.3449044,1.1538066)
				\psline[linecolor=black, linewidth=0.04, arrowsize=0.05291667cm 3.96,arrowlength=1.63,arrowinset=0.0]{<-}(2.596967,0.82170624)(3.086921,1.7078539)(3.086921,1.7078539)
				\psline[linecolor=black, linewidth=0.04, arrowsize=0.05291667cm 3.96,arrowlength=1.63,arrowinset=0.0]{<-}(3.2869625,1.7155086)(2.7970085,0.8293609)(2.7970085,0.8293609)
				\psline[linecolor=black, linewidth=0.04](3.6396587,0.20695473)(3.6230698,-0.5929709)
				\psline[linecolor=black, linewidth=0.04](3.798846,0.20659028)(3.7822573,-0.5933354)
				\psline[linecolor=black, linewidth=0.04, arrowsize=0.05291667cm 3.96,arrowlength=1.63,arrowinset=0.0]{<-}(3.1261642,-0.10421068)(4.138715,-0.09688591)(4.138715,-0.09688591)
				\psline[linecolor=black, linewidth=0.04, arrowsize=0.05291667cm 3.96,arrowlength=1.63,arrowinset=0.0]{<-}(4.2434444,-0.26749352)(3.2308936,-0.27481827)(3.2308936,-0.27481827)
				\rput[bl](4.0454454,-0.7942492){$\hat{\vec{x}}_{1}$}
				\rput[bl](2.6903434,1.7118733){$\hat{\vec{d}}_{2}$}
				\rput[bl](3.4209557,1.2179956){$\hat{\vec{x}}_{2}$}
				\rput[bl](3.5229964,-2.1166983){$\hat{\vec{x}}_{N}$}
				\rput[bl](3.8658535,-1.5044533){$\hat{\vec{d}}_{N}$}
				\end{pspicture}
			}
		\end{minipage}
	}
	\subfloat[Inhomogeneity $\tau_{m}$]{\label{Location1-2}\centering
		\begin{minipage}[c][4.5cm][c]{.45\linewidth}\centering
			\psscalebox{0.9 0.9} 
			{\centering\small
				\begin{pspicture}(0,-0.97)(2.52,0.97)
				\definecolor{colour0}{rgb}{1.0,0.8,0.0}
				\pscircle[linecolor=black, linewidth=0.04, fillstyle=solid,fillcolor=colour0, dimen=outer](0.97,0.0){0.97}
				\psline[linecolor=black, linewidth=0.04, linestyle=dashed, dash=0.17638889cm 0.10583334cm, arrowsize=0.01cm 4.0,arrowlength=1.4,arrowinset=0.0,dotsize=0.07055555cm 2.0]{<-*}(1.58,0.59)(0.98,0.01)(0.98,0.01)
				\rput[bl](0.78,-0.47){$\vec{r}_{m}$}
				\rput[bl](1.22,-0.05){$\alpha_m$}
				\rput[bl](2.02,-0.63){$\vec{D}_{m}$}
				\end{pspicture}
			}		
		\end{minipage}
	}
	\caption{\label{Location}Configuration of the scattering problem for $M=3$ (left) and sketch of the inhomogeneity $\tau_{m}$ (right).}
\end{figure}
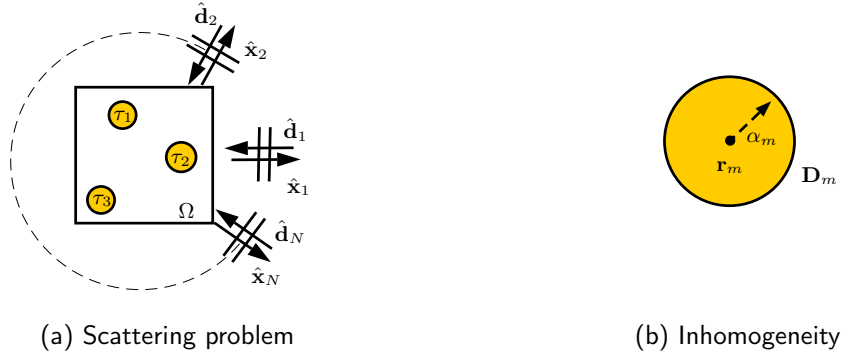

Here, we assume that all materials are non-magnetic ($\mu(\fx)\equiv\mu_0=\SI{1.256e-6}{\henry/\meter}$) and characterized by their dielectric permittivity at angular frequency $\omega=2\pi f$, $f$ being the frequency. Let us denote $\eps_m$ and $\eps_0$ the value of electrical permittivity of $\tau_m$ and $\mathbb{R}^2$, respectively. In so doing the following piecewise constant function can be introduced:
\[
\eps(\fx):=\left\{
\begin{array}{rcl}
\eps_{m}&\mbox{for}&\fx\in\tau_{m},\\
\eps_{0}&\mbox{for}&\mathbb{R}^{2}\backslash\overline{\tau}.
\end{array}
\right.\]
Let $k_0=\omega\sqrt{\eps_0\mu_0}=2\pi/\lambda$ be the wavenumber with positive wavelength $\lambda$, satisfying $\alpha_m\sqrt{\eps_{m}/\eps_{0}} \ll \lambda/2$ for all $m=1,2,\cdots,M$ (refer to \cite{SKL}).

In this contribution, we consider the plane-wave illumination: let $u^{i}(\fx)=\e^{\ci k\fd\cdot\fx}$, $\fx\in\mathbb{R}^2$ be an incident field with direction of propagation $\fd\in \mathbb{S}^{1}$, where $\mathbb{S}^{1}$ denotes the two-dimensional unit circle. Let $u(\fx,\fd)$ be the time-harmonic total field that satisfies the Helmholtz equation 
$$\lap u(\fx,\fd)+ \omega^2\mu_0\eps(\fx) u(\fx,\fd)=0$$
with transmission conditions at  boundaries $\pa\tau_m$. It is well-known that the total field can be written as the sum of the incident field $u^{i}(\fx,\fd)$ and the scattered field $u^{s}(\fx,\fd)$, where $u^{s}(\fx,\fd)$  satisfies the Sommerfeld radiation condition
\[
\lim_{|\fx|\to\infty}\sqrt{|\fx|}\left(\f{\pa u^{s}(\fx,\fd)}{\pa|\fx|}-\ci k_{0}u^{s}(\fx,\fd)\right)=0\]
uniformly into all directions $\xh=\fx\slash|\fx|$. We denote $u_{\infty}(\xh,\fd)$ the far-field pattern of $u^{s}(\fx,\fd)$ defined on $\mathbb{S}^{1}$ that satisfies
\[
u^{s}(\fx,\fd)=\frac{\e^{\ci k_{0}\fd\cdot\fx}}{\sqrt{|\fx|}}\left[u_{\infty}(\xh,\fd)+\mathcal{O}\left(\frac{1}{|\fx|}\right)\right]\]
uniformly into all directions $\xh=\fx\slash|\fx|$ and $|\fx|\longrightarrow\infty$. 
Based on \cite{AmmariKang}, the asymptotic expansion formula of $u_{\infty}(\xh,\fd)$ can be written as  .
\begin{equation}\label{asymptotic} u_{\infty}(\xh,\fd)=\frac{k_{0}^{2}(1+\ci)}{4\sqrt{k_{0}\pi}}\sum_{m=1}^{M}\alp_{m}^{2}\left(\frac{\eps_{m}-\eps_{0}}{\sqrt{\eps_{0}\mu_{0}}}\right)|\fD_{m}|\\
\mathrm{e}^{\ci k_{0}(\fd-\xh)\cdot\fr_{m}}
 + \mathcal{O}(\alp_{m}^{2})
\end{equation}
which plays a key role of the theoretical analysis of indicator function of DSM in mono-static configuration introduced in Section \ref{sec:4}

\section{Introduction of direct sampling method}\label{sec:3}
According to \cite{dsm2d_farfield}, the indicator function of the classical DSM with a set of measured far-field pattern data $\mathcal{F}=\{u_{\infty}(\xh_n,\fd): n=1,2,\cdots,N\}$ for a fixed incident direction $\fd$ is defined by 
\begin{equation}\label{DSM}	
	\mathcal{I}_{\mathrm{DSM}}(\fz,\fd):=\f{|\bkc{u_{\infty}(\xh_n,\fd),\e^{-\ci k_0\xh_{n}\cdot\fz}}_{L^{2}(\mathbb{S}^1)}|}{\|u_{\infty}(\xh_n,\fd)\|_{L^{2}(\mathbb{S}^1)}}
\end{equation}
where
\begin{align*}
\bkc{a(\xh_n),b(\xh_n)}_{L^{2}(\mathbb{S}^1)}&:=\sum_{n=1}^{N}a(\xh_n)\overline{b(\xh_n)}\\
\norm{a(\xh_n)}_{L^{2}(\mathbb{S}^1)}^2&:=\bkc{a(\xh_n),a(\xh_n)}_{L^{2}(\mathbb{S}^1)}.
\end{align*}
Based on \cite[Theorem 4.1]{InverseProblems-Kang-2017}, $\mathcal{I}_{\mathrm{DSM}}(\fz, \fd )$ can be represented by
\[\mathcal{I}_{\mathrm{DSM}}(\fz, \fd)=\frac{|\Psi_1(\fz, \fd)|}{\displaystyle\max_{\fz\in\Omega}|\Psi_1(\fz, \fd)|},\]
where
\begin{equation}\label{Structure_case1}
	\Psi_1(\fz, \fd)=\sum_{m=1}^{M}\alp_{m}^{2}(\eps_{m}-\eps_{0})\e^{\ci k_{0}\fd\cdot\fr_{m}}\BesselJ_{0}(k_0|\fz-\fr_{m}|).
\end{equation}
Here, $\BesselJ_{0}$ denotes the Bessel function of order zero of the first kind. Thanks to \eqref{Structure_case1} we can observe that $\mathcal{I}_{\mathrm{DSM}}(\fz)$ exhibits a maximum when $\fz=\fr_m$ and $0<\mathcal{I}_{\mathrm{DSM}}(\fz)<1$ at $\fz\notin\tau$ so that the location $\fr_m$ of $\tau_m$ can be identified.

In the multiple impinging case $\left(\fd_{l}, l=1,2,\cdots,L\right)$, $L$ being the number of incident directions, the indicator function of DSM is defined by 
\begin{equation}
\mathcal{I}_{\mathrm{DSM}}(\fz;k_{0}):=\max\{\mathcal{I}_{\mathrm{DSM}}(\fz;\fd_{1},k_{0}),\mathcal{I}_{\mathrm{DSM}}(\fz;\fd_{2},k_{0}),\cdots,\mathcal{I}_{\mathrm{DSM}}(\fz;\fd_{L},k_{0})\}\label{DSM2}
\end{equation}
Note that \eqref{DSM}  and \eqref{DSM2} are equivalent when $L=1$.

\section{Analysis and improvement of direct sampling method in mono-static configuration}\label{sec:4}
Let us now deal with the monostatic configuration  in which an antenna acts as receiver and transmitter, implying $\fd_{n}=-\xh_{n}$, and is moved from place to place giving a set of measured far-field pattern data defined by $\mathcal{M}=\{u_{\infty}(\xh_n,\fd_{n}): n=1,2,\cdots,N\}$   

As examplified in \cite{bektas2016direct}, the direct sampling method in such a configuration failed to provide a proper localisation of the defects (see also Fig. \ref{Result1-2}) when using the  indicator function $\mathcal{I}_{\mathrm{DSM}}^{\mathrm{mono}}(\fz)$ directly deduced from \eqref{DSM} and defined as
\begin{equation}\label{MonoDSM}	
	\mathcal{I}_{\mathrm{DSM}}^{\mathrm{mono}}(\fz):=\f{|\bkc{u_{\infty}(\xh_n,\fd_{n}),\e^{-\ci k_0\xh_{n}\cdot\fz}}_{L^{2}(\mathbb{S}^1)}|}{\|u_{\infty}(\xh_n,\fd_{n})\|_{L^{2}(\mathbb{S}^1)}}.
\end{equation}
In  \cite{bektas2016direct}, a modified indicator involving a heuristic factor is proposed to solve the problem, yet no theoretical explanation is provided. 
In the following the theoretical reason of this miss-localization is exhibited and a modified version of the DSM is introduced. Let us analyze the indicator function $\mathcal{I}_{\mathrm{DSM}}^{\mathrm{mono}}(\fz)$ to explain the inaccurate localization in the mono-static configuration.
\begin{theorem}\label{thm_DSM}
	Assume that the total number $N$ of incident and observation directions is sufficiently large. Then, $\mathcal{I}_{\mathrm{DSM}}^{\mathrm{mono}}(\fz)$ can be represented as:
	
	\[\mathcal{I}_{\mathrm{DSM}}^{\mathrm{mono}}(\fz)\approx\frac{|\Psi_{1}(\fz)|}{\displaystyle\max_{\fz\in\Omg}|\Psi_{1}(\fz)|},\]
	where
	\begin{equation}\label{result1}
		\Psi_{1}(\fz)=\sum_{m=1}^{M}\alp_{m}^{2}(\eps_{m}-\eps_{0})|\fD_{m}|\BesselJ_{0}(k_{0}|2\fr_{m}-\fz|).
	\end{equation}
\end{theorem}
\begin{proof}
If $N$ is sufficiently large, the following relation holds for $\fz\in\mathbb{R}^2$ (see \cite{Park})
  \begin{equation}\label{RelationBessel}
 \dfrac{1}{N}\sum_{n=1}^{N}\e^{-\ci k_0\xh_{n}\cdot\fz}\approx\dfrac{1}{2 \pi}\int_{\mathbb{S}^{1}}\e^{-\ci k_0\xh\cdot\fz}d\xh=\BesselJ_{0}(k_0|\fz|).
  \end{equation}
Since $\fd_n=-\xh_n$, applying \eqref{asymptotic} and \eqref{RelationBessel} to \eqref{DSM}, we can evaluate
	\begin{align*}
	\bkc{u_{\infty}(\xh_n,\fd_n),\e^{-\ci k_0\xh_{n}\cdot\fz}}_{L^{2}(\mathbb{S}^1)}
	&\approx\f{k^{2}(1+\ci)}{4\sqrt{k\pi}}\sum_{m=1}^{M}\alp_{m}^{2}\left(\frac{\eps_{m}-\eps_{0}}{\sqrt{\eps_{0}\mu_{0}}}\right)\abs{\fD_{m}}\bke{\sum_{n=1}^{N}\e^{-\ci k\xh_{n}\cdot\bke{2\fr_{m}-\fz} }}\\
	&\approx \f{k^{2}(1+\ci)\pi}{2\sqrt{k\pi}}\sum_{m=1}^{M}\alp_{m}^{2}\left(\frac{\eps_{m}-\eps_{0}}{\sqrt{\eps_{0}\mu_{0}}}\right)\abs{\fD_{m}}\BesselJ_{0}(k_0|2\fr_{m}-\fz|).
	\end{align*}
Finally, applying H\"older's inequality
\begin{equation*}
|\bkc{u_{\infty}(\xh_n,\fd),\e^{-\ci k_0\xh_{n}\cdot\fz}}_{L^{2}(\mathbb{S}^1)}| \leq\|u_{\infty}(\xh_n,\fd)\|_{L^{2}(\mathbb{S}^1)}
\end{equation*}
	leads to \eqref{result1} which completes the proof.
\end{proof}

The structure of \eqref{result1} explains that DSM within the mono-static configuration is no longer proportional to $\abs{\BesselJ_{0}(k_{0}|\fr_{m}-\fz|)}$ but to $\abs{\BesselJ_{0}(k_{0}|2\fr_{m}-\fz|)}$. This means that $\mathcal{I}_{\mathrm{DSM}}^{\mathrm{mono}}(\fz)$ reaches its maximum value at shifted locations $\fz=2\fr_{m}$. Due to this reason, traditional application of DSM will lead to miss-localization of the inhomogeneities.

Thanks to \eqref{result1}, an alternative indicator function of DSM $\mathcal{I}_{\mathrm{MDSM}}^{\mathrm{mono}}(\fz)$ can be proposed: for $\fz\in\Omega$, 
\begin{equation}
\mathcal{I}_{\mathrm{MDSM}}^{\mathrm{mono}}(\fz):=\f{|\bkc{u_{\infty}(\xh_n,\fd_n),\e^{-2ik_0\xh_{n}\cdot\fz}}_{L^{2}(\mathbb{S}^1)}|}{\|u_{\infty}(\xh_n,\fd_n)\|_{L^{2}(\mathbb{S}^1)}}.
\label{MDSM}
\end{equation}
Following the same path (omitted here) than for Theorem \ref{thm_DSM} leads to
\begin{theorem}\label{thm_MDSM}
	Assume that the total number $N$ of incident and observation directions is sufficiently large. Then, $\mathcal{I}_{\mathrm{MDSM}}^{\mathrm{mono}}(\fz)$ can be represented as:
	\[\mathcal{I}_{\mathrm{MDSM}}^{\mathrm{mono}}(\fz)=\frac{|\Psi_{2}|}{\displaystyle\max_{\fz\in\Omg}|\Psi_{2}|},\]
	where
	\begin{equation}\label{MDSM}
		\Psi_{2}(\fz)=\sum_{m=1}^{M}\alp_{m}^{2}(\eps_{m}-\eps_{0})|\fD_{m}|\BesselJ_{0}(2k_{0}|\fr_{m}-\fz|).
	\end{equation}
\end{theorem}

As shown in \eqref{MDSM}  $\mathcal{I}_{\mathrm{MDSM}}^{\mathrm{mono}}(\fz)$ is proportional to $\abs{\BesselJ_{0}(2k_{0}|\fr_{m}-\fz|)}$ which, on the contrary of \eqref{result1}, has its maximum values at $\fz=\fr_m$, $m=1,2,\cdots,M$, which corresponds to the localization of the defects to be identified. It is interesting to observe that, according to \cite{dsm2d_ito1,Kang_mfDSM,dsm2d_farfield,InverseProblems-Kang-2017,PARK201863,PARK2018648}, the traditional DSM in the multi-static configuration is proportional to $\abs{\BesselJ_{0}(k_{0}|\fr_{m}-\fz|)}$. By comparing the oscillation property of $\BesselJ_{0}(k_0|x|)$ and $\BesselJ_{0}(2k_0|x|)$, it can be shown that $\mathcal{I}^{\mathrm{mono}}_{\mathrm{MDSM}}(\fz)$ will contain more artifacts than $\mathcal{I}_{\mathrm{DSM}}(\fz)$.

\section{Numerical experiments}\label{sec:5}
Numerical experiments are provided to support the results presented in Theorem \ref{thm_DSM} and \ref{thm_MDSM}. For the simulation, a fixed frequency $f=c_0/\lambda\approx\SI{749.481}{\MHz}$ where $c_{0}=1\slash\sqrt{\eps_0\mu_0}$ is the speed of light and $\lambda=\SI{0.4}{\meter}$ is considered. The number of incident and observation directions is set to $N=36$, the latter being uniformly distributed on $\mathbb{S}^1$ except  stated otherwise. We set $\Omg$ as a square of side length $4\lambda$ uniformly discretized with $50\times50$ square pixels. The far-field patterns $u_\infty(\xh_{n},\fd_{n})$ are generated via \texttt{FEKO} (EM simulation software), where
\[\xh_{n}=\left(\cos\frac{2\pi(n-1)}{N},\sin\frac{2\pi(n-1)}{N}\right)\]
with $N=36$. A \SI{20}{\dB} white Gaussian random noise is added to unperturbed data using \texttt{MATLAB} function \texttt{awgn} included in the signal processing package.

To compare the accuracy of the results  the Jaccard index \cite{jaccard} is used. It measures the similarity of two finite sample sets $A$ and $B$ and is defined as
\[{\mathcal{J}(A,B)}(\%):=\frac{|A\cap B|}{|A\cup B|}\times 100.\]
In our work, the Jaccard index is calculated by comparing $\mathcal{I}_{\mathrm{EXACT}}(\fz)$ with an index map $\mathcal{I}^{\kappa}(\fz)$ defined for a threshold $\kappa\in[0,1]$ as
\[
\mathcal{I}_{\mathrm{EXACT}}(\fz):=
\left\{
\begin{array}{lcl}
1&\mbox{for}&\fz\in\tau\\
0&\mbox{for}&\fz\in\mathbb{R}^{2}\backslash\overline{\tau},
\end{array}
\right.\]
and
\[\mathcal{I}^{\kappa}(\fz):=
\left\{
\begin{array}{ccl}
\mathcal{I}(\fz)&\mbox{if}&\mathcal{I}(\fz)\geq\kappa\\
0&\mbox{if}&\mathcal{I}(\fz)<\kappa,
\end{array}
\right.\]
respectively. Here, $\mathcal{I}(\fz)$ is either $\mathcal{I}_{\mathrm{DSM}}(\fz)$ (\ref{DSM2}),  $\mathcal{I}^{\mathrm{mono}}_{\mathrm{DSM}}(\fz)$ (\ref{MonoDSM}) or $\mathcal{I}^{\mathrm{mono}}_{\mathrm{MDSM}}(\fz)$ (\ref{MDSM}). Since for a fixed threshold,  the Jaccard index is made of a numerator which measures the common portion of the reconstructed and the exact images and a denominator which is the reunion of the portion of the reconstructed and the exact images, the more artefacts there are the higher the denominator is and the lower the Jaccard index is, so is the similarity between the two images

For each example the map of the indicator function is presented in the multi-static case (\ref{DSM2}) using the $N^2$ collected data and in the mono-static case  using the $N$ collected data thanks to either (\ref{MonoDSM}) or (\ref{MDSM}). 

\begin{example}[Small disks of same radii and permittivity]\label{EX:1}
	First, we consider small dielectric disks $\tau_m$ with  $\alpha_m\equiv0.075\lambda$ and $\eps_m\equiv5\eps_0$, $m=1,2,3$. The locations $\fr_m$ of $\tau_m$ are
	$\fr_1=(0.75\lambda,-0.75\lambda)$, $\fr_2=(-\lambda,-0.5\lambda)$, and $\fr_3=(-0.75\lambda,\lambda)$. According to the results in Fig.~\ref{Result1-2}, the location of $\fr_m\in\tau_m$ can be identified using the classical DSM indicator function $\mathcal{I}_{\mathrm{DSM}}(\fz)$ (\ref{DSM2}) when using the multi-static data (Fig.~\ref{Ex1ClassicalMultipleDSM}) but failed when using the mono-static ones (Fig.~\ref{Ex1ClassicalMonoDSM}) whereas more accurate locations are retrieved via the map of $\mathcal{I}_{\mathrm{MDSM}}^{\mathrm{mono}}(\fz)$ (Fig.~\ref{Ex1MonoMDSM}); however, due to the intrinsic lack of information of the monostatic configuration, only two of the three defects are properly identified.  As expected in the mono-static configuration a number of artifacts is also included in the map as discussed at the end of \S~\ref{sec:4}.
\end{example}

\begin{figure*}
	\centering
	\subfloat[\label{Ex1ClassicalMultipleDSM}$\mathcal{I}_{\mathrm{DSM}}(\fz)$]{\centering\includegraphics[width=.225\textwidth]{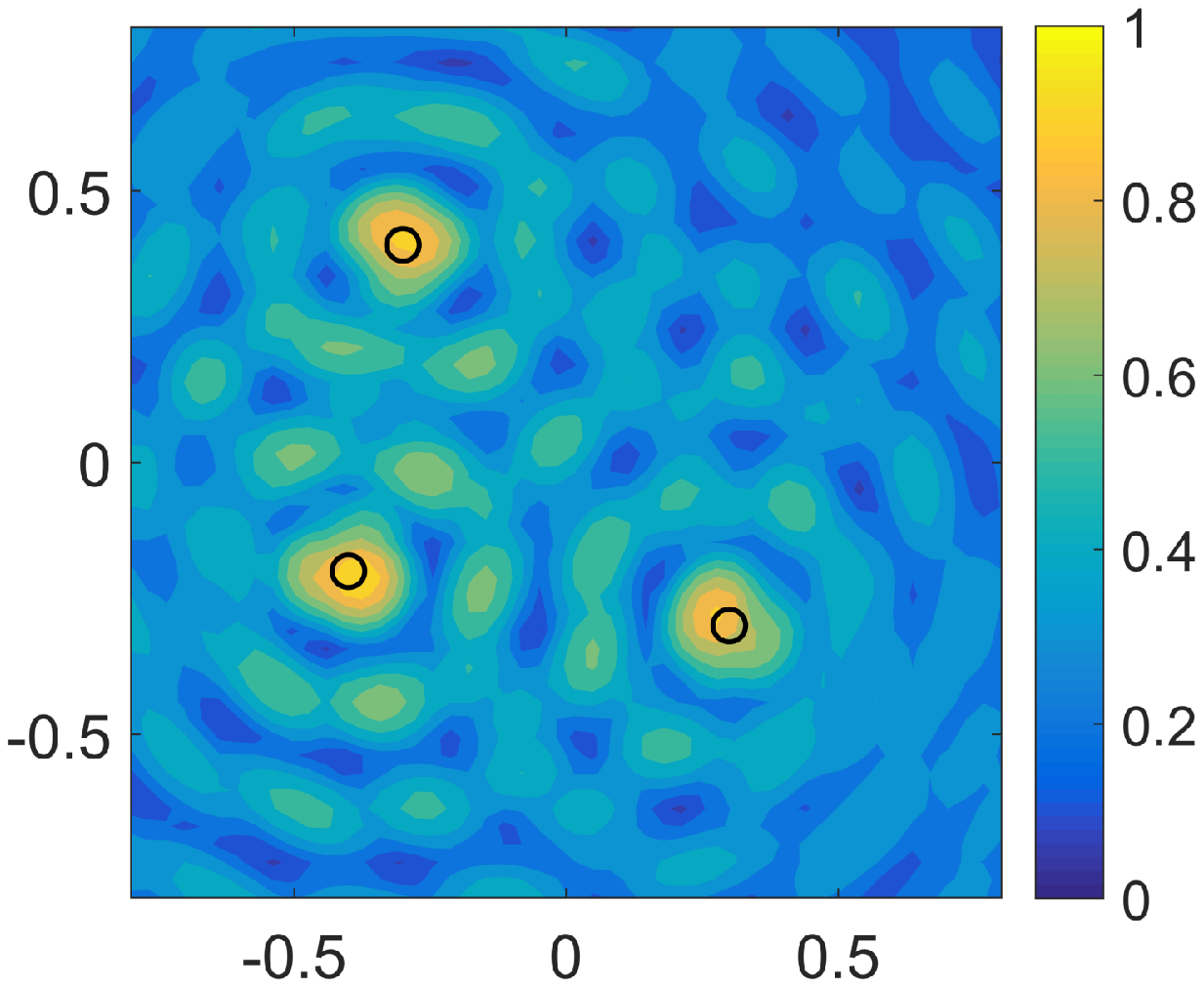}}\hspace*{0.01\textwidth}
	\subfloat[\label{Ex1ClassicalMonoDSM}$\mathcal{I}^{\mathrm{mono}}_{\mathrm{DSM}}(\fz)$]{\centering\includegraphics[width=.225\textwidth]{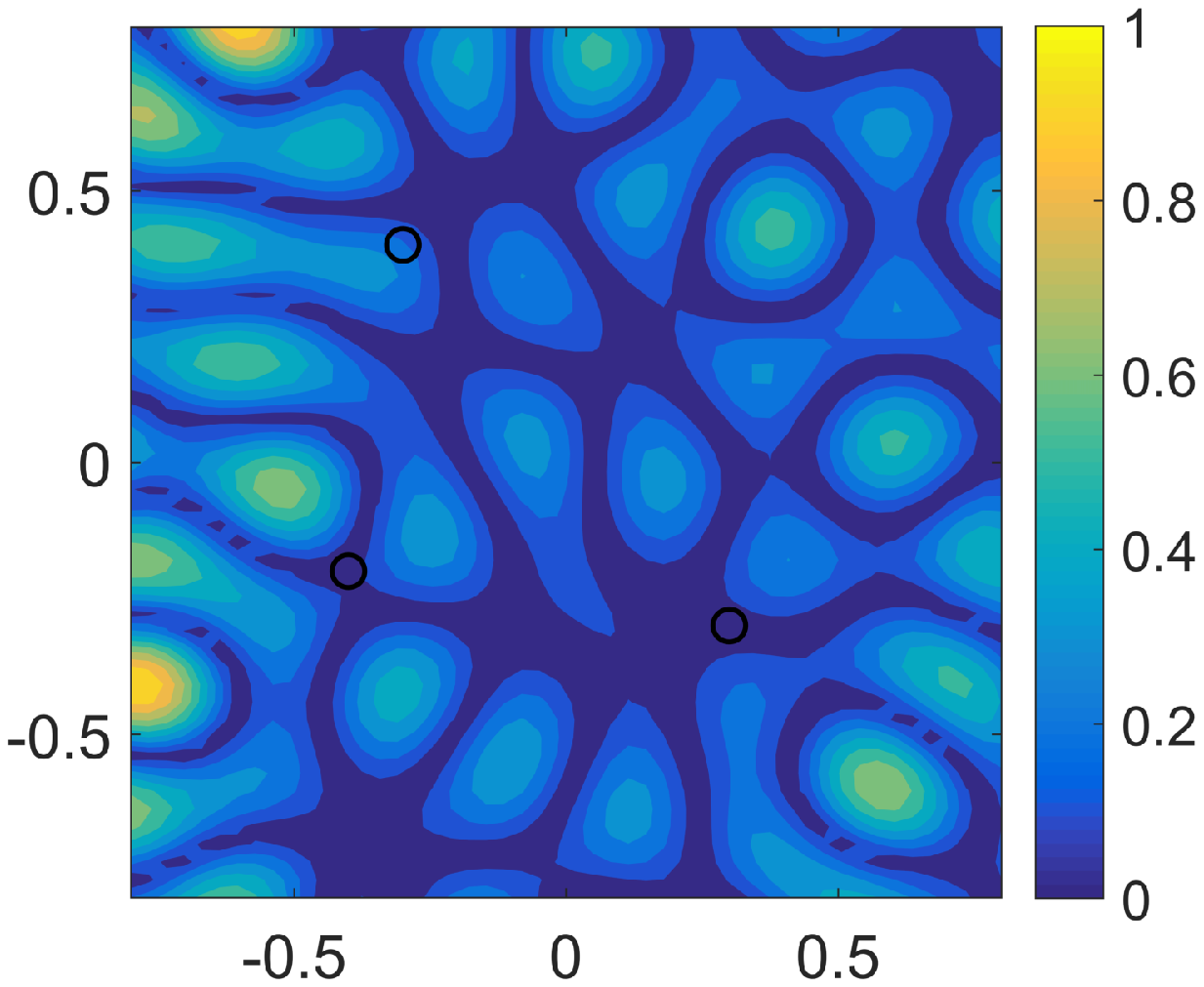}}\hspace*{0.01\textwidth}
	\subfloat[\label{Ex1MonoMDSM}$\mathcal{I}^{\mathrm{mono}}_{\mathrm{MDSM}}(\fz)$]{\centering\includegraphics[width=.225\textwidth]{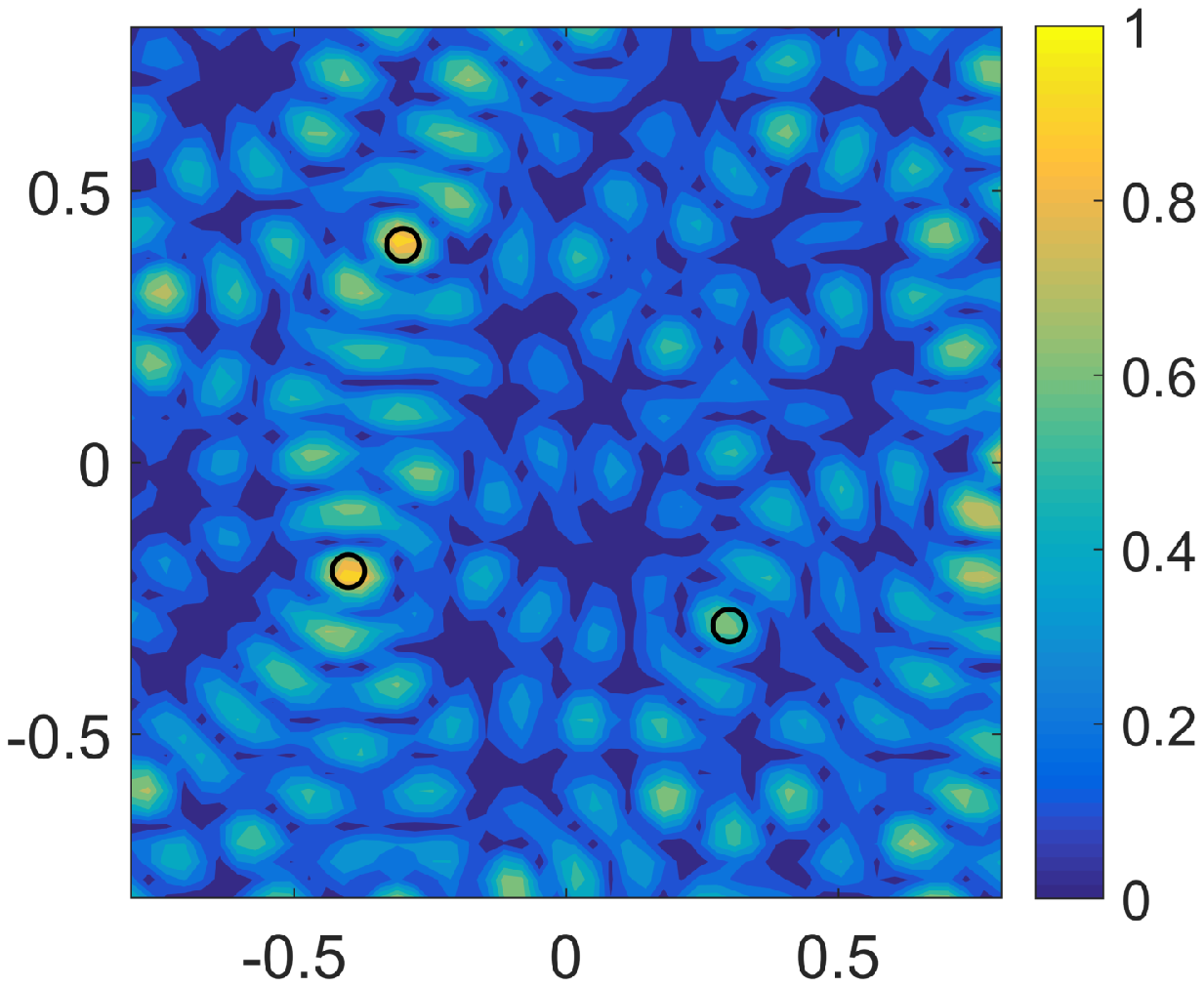}}\hspace*{0.01\textwidth}
	\subfloat[Jaccard index]{\centering\includegraphics[width=.225\textwidth]{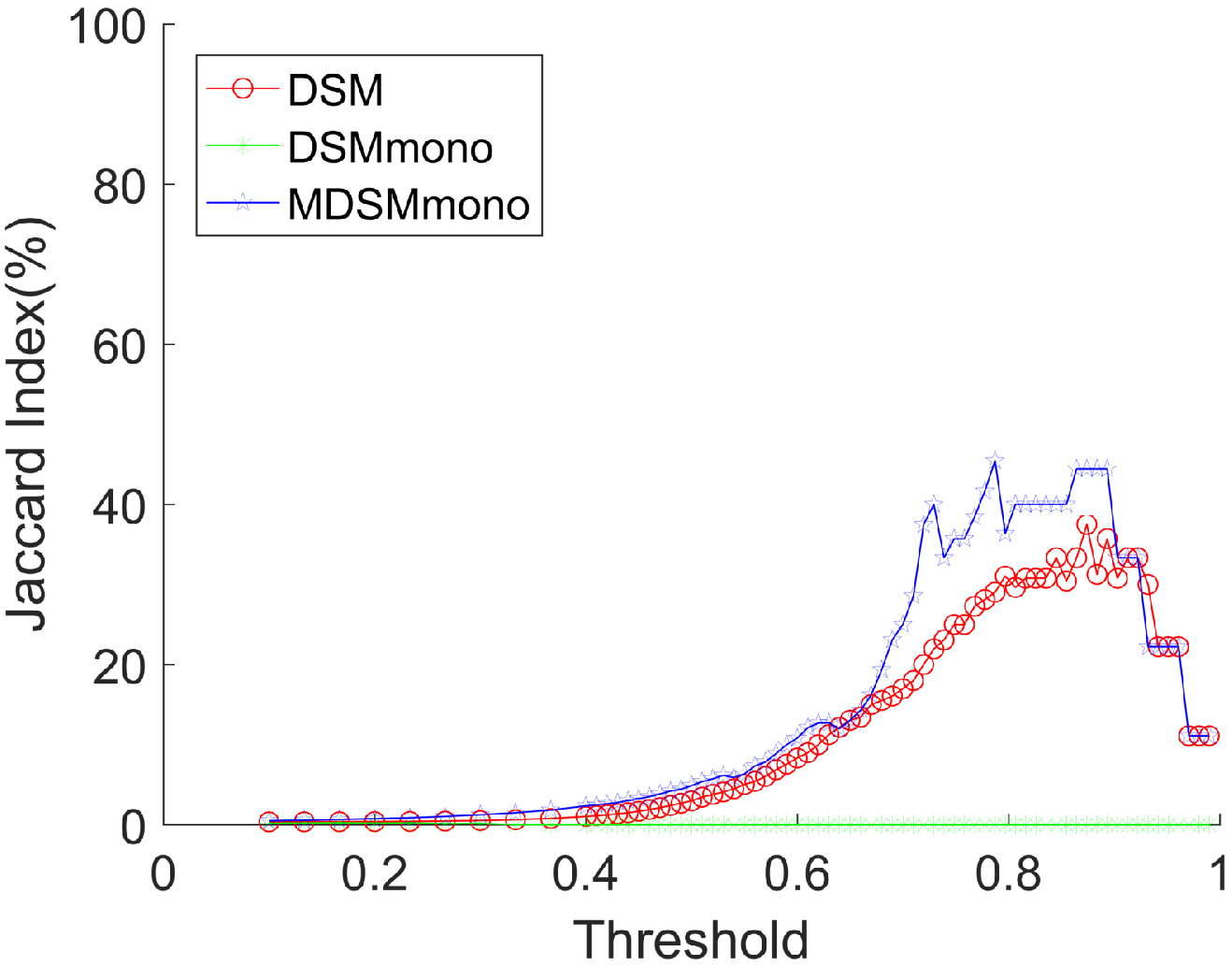}}
	\caption{\label{Result1-2}Simulation results of Example \ref{EX:1}}
\end{figure*}


\begin{example}[Large disk]\label{EX:2}
	In order to verify that our approach still behaves properly
	when the small obstacle hypothesis is no longer verified, we are considering the identification of an extended target designed as a single disk circle $\tau$ located at $\fr=(-0.75\lambda,-0.75\lambda)$ with radius $\alpha\equiv1\lambda$ and permittivity $\eps=5\eps_0$. Here also the shifting problem occurs in $\mathcal{I}_{\mathrm{DSM}}^{\mathrm{mono}}(\fz)$ as shown in Fig.~\ref{BigObstacleMonoDSM} whereas, when using  $\mathcal{I}_{\mathrm{MDSM}}^{\mathrm{mono}}(\fz)$, a better localisation of the center of target is obtained (Fig.~\ref{BigObstacleMonoModifiedDSM}) even if none of them is able to estimate neither the shape nor the size of the defect. As expected  better results are obtained when using the mutli-static data (Fig.~\ref{BigObstacleDSM}).
\end{example}
\begin{figure*}
	\centering
	\subfloat[\label{BigObstacleDSM}$\mathcal{I}_{\mathrm{DSM}}(\fz)$]{\centering\includegraphics[width=.225\textwidth]{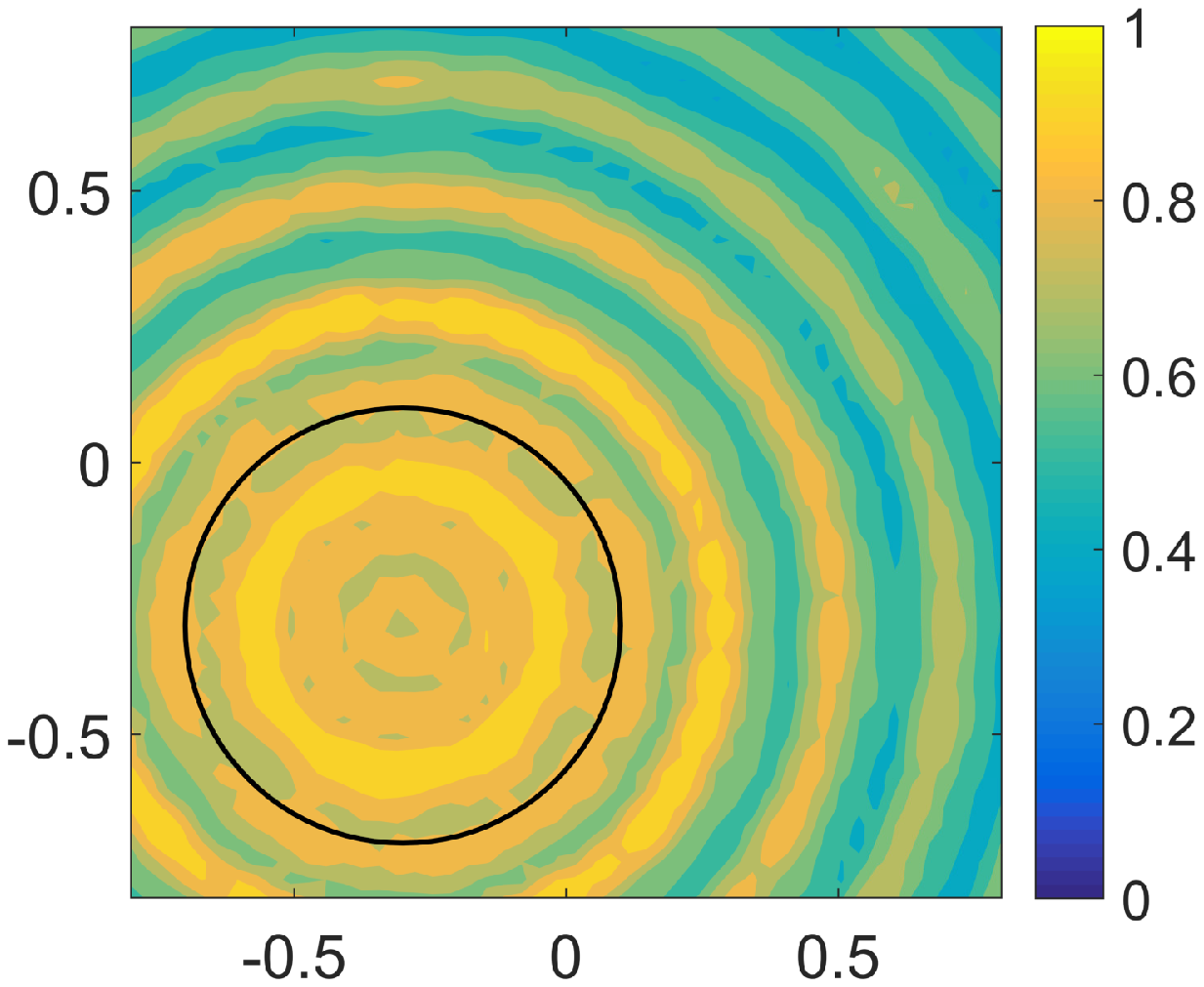}}\hspace*{0.01\textwidth}
	\subfloat[\label{BigObstacleMonoDSM}$\mathcal{I}^{\mathrm{mono}}_{\mathrm{DSM}}(\fz)$]{\centering\includegraphics[width=.225\textwidth]{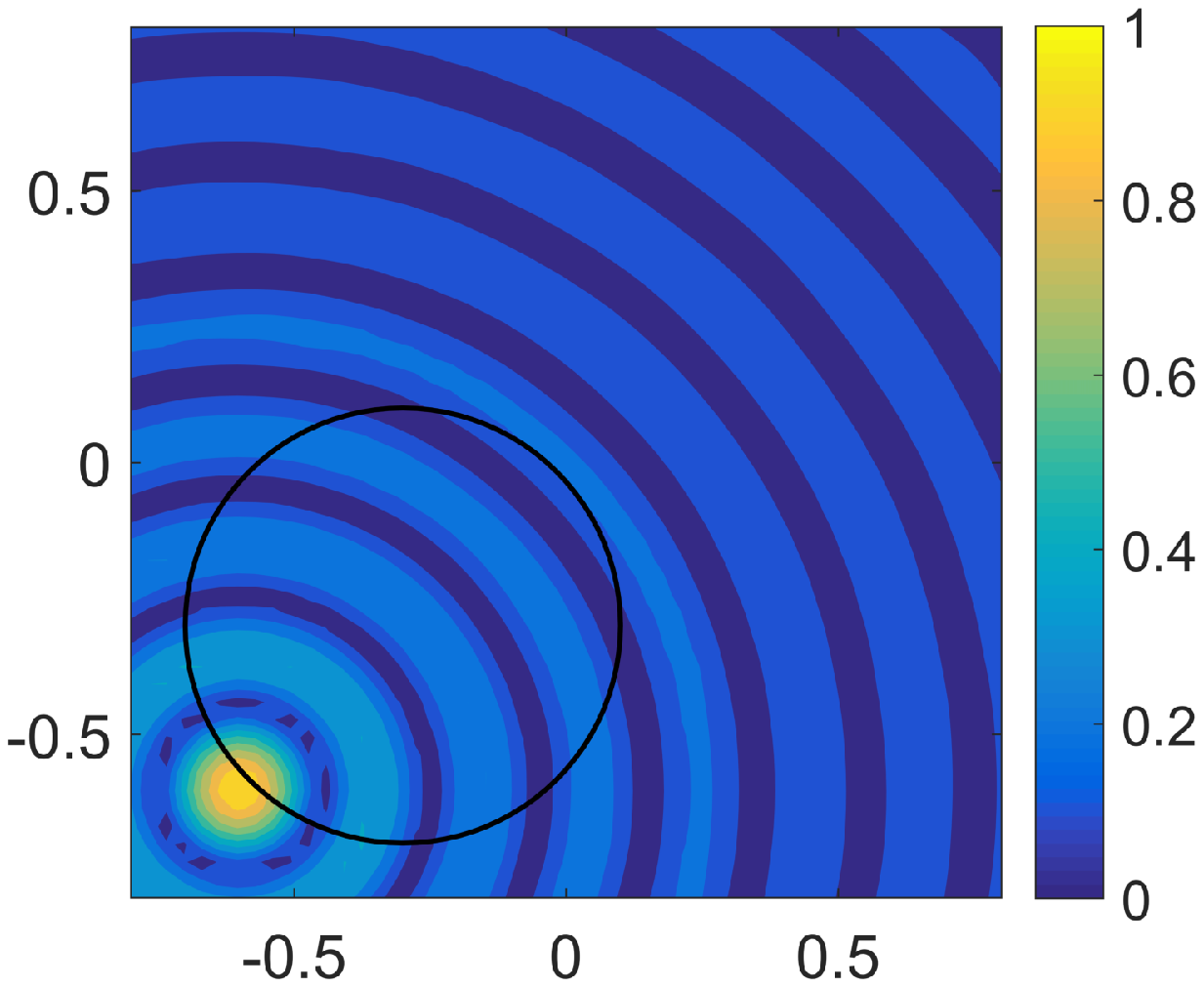}}\hspace*{0.01\textwidth}
	\subfloat[\label{BigObstacleMonoModifiedDSM}$\mathcal{I}^{\mathrm{mono}}_{\mathrm{MDSM}}(\fz)$]{\centering\includegraphics[width=.225\textwidth]{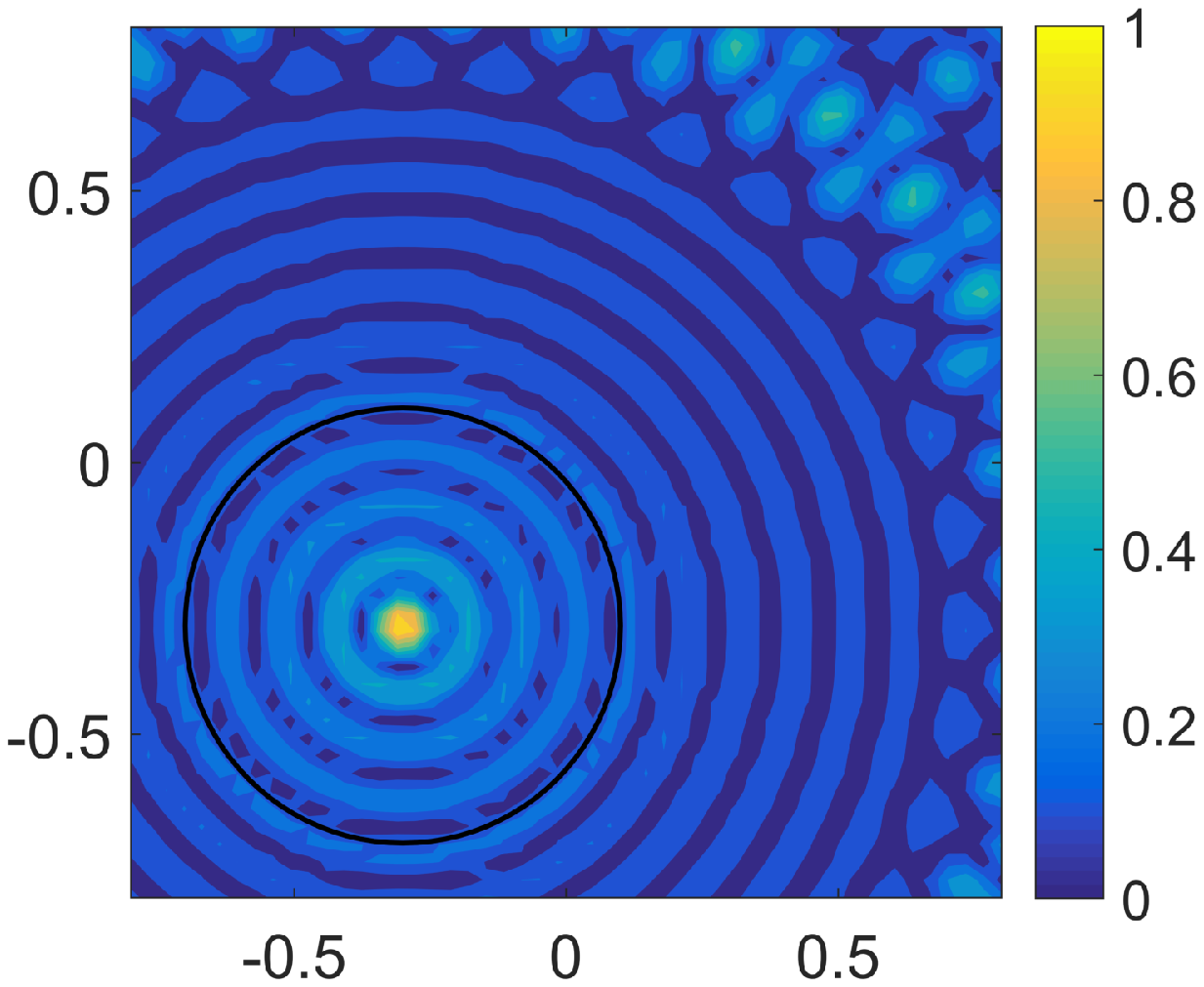}}\hspace*{0.01\textwidth}
	\subfloat[Jaccard index]{\centering\includegraphics[width=.225\textwidth]{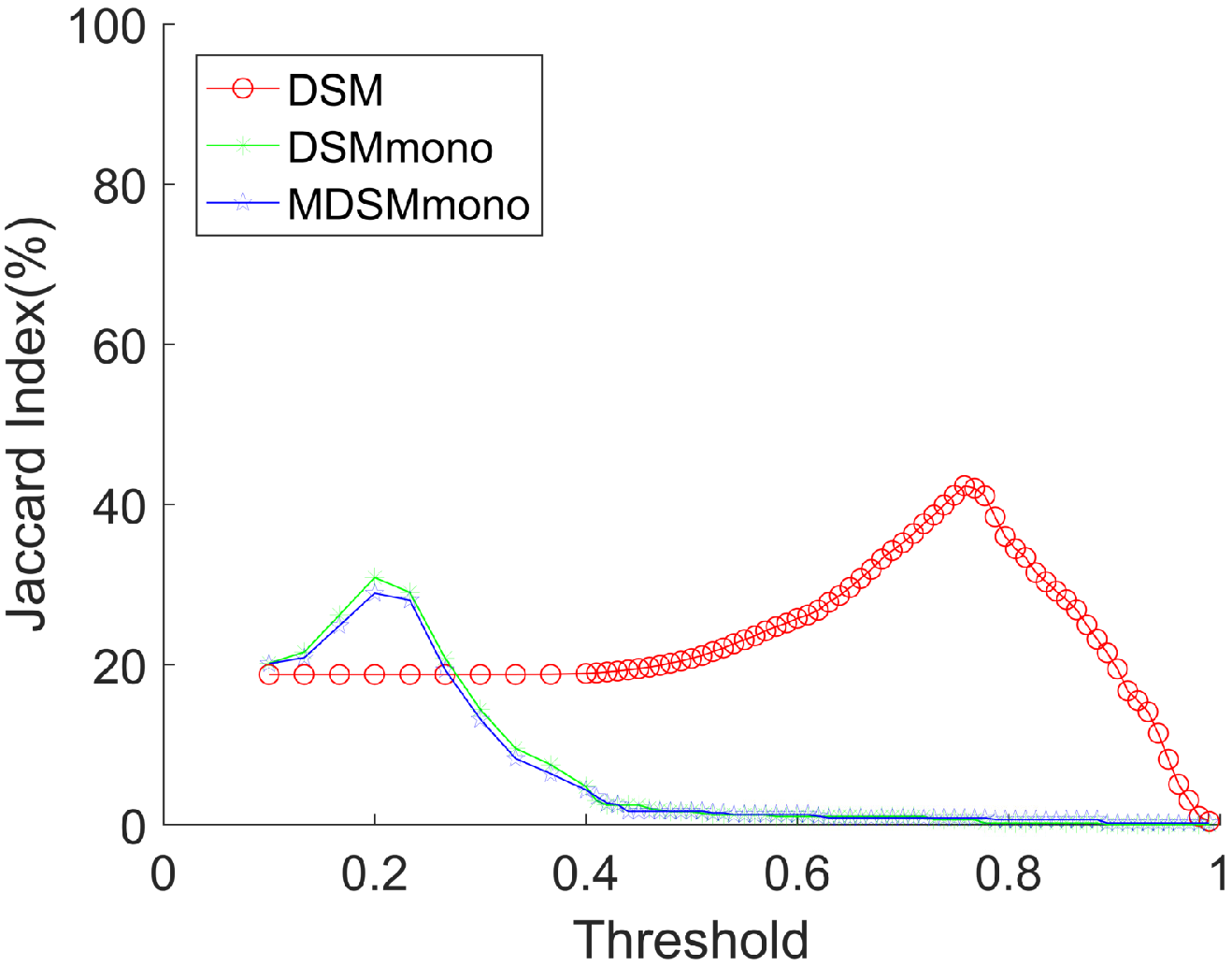}}
	\caption{\label{Result2-2}Simulation results of Example \ref{EX:2}}
\end{figure*}

\begin{example}[Limited view]\label{EX:3}
	Motivated by the application in GPR and SAR, we apply the designed indicator function $\mathcal{I}_{\mathrm{MDSM}}^{\mathrm{mono}}(\fz)$ when the range of incident and observation directions is limited. It is important to emphasize that due to the use of the far-field hypothesis such a configuration is not directly related to a GPR configuration, even if the influence of the limited aspect of the data is exemplified.
		 
	The configuration is the same as for Example \ref{EX:1} except the range of incident and observation directions which is limited to the upper half-circle with only $N=19$ collected far-field data. The simulation results are displayed in Fig. \ref{Result1-3}. As for the two previous examples the results using the multi-static scattered field provide the best localisations (Fig.~\ref{LimitedViewEx1ClassicalMultipleDSM}) whereas the mono-static case using the classical DSM does not provide any good results since the shifting problem still occurs (Fig.~\ref{LimitedViewEx1ClassicalMonoDSM}). As expected the mono-static modified DSM is able to localize two obstacles among the three   (Fig.~\ref{LimitedViewEx1MonoMDSM}) as it was the case with full-view aperture (Fig.~\ref{Ex1MonoMDSM}).
\end{example}


\begin{figure*}
	\centering
	\subfloat[\label{LimitedViewEx1ClassicalMultipleDSM}Map of $\mathcal{I}_{\mathrm{DSM}}(\fz)$]{\centering \includegraphics[width=.225\textwidth]{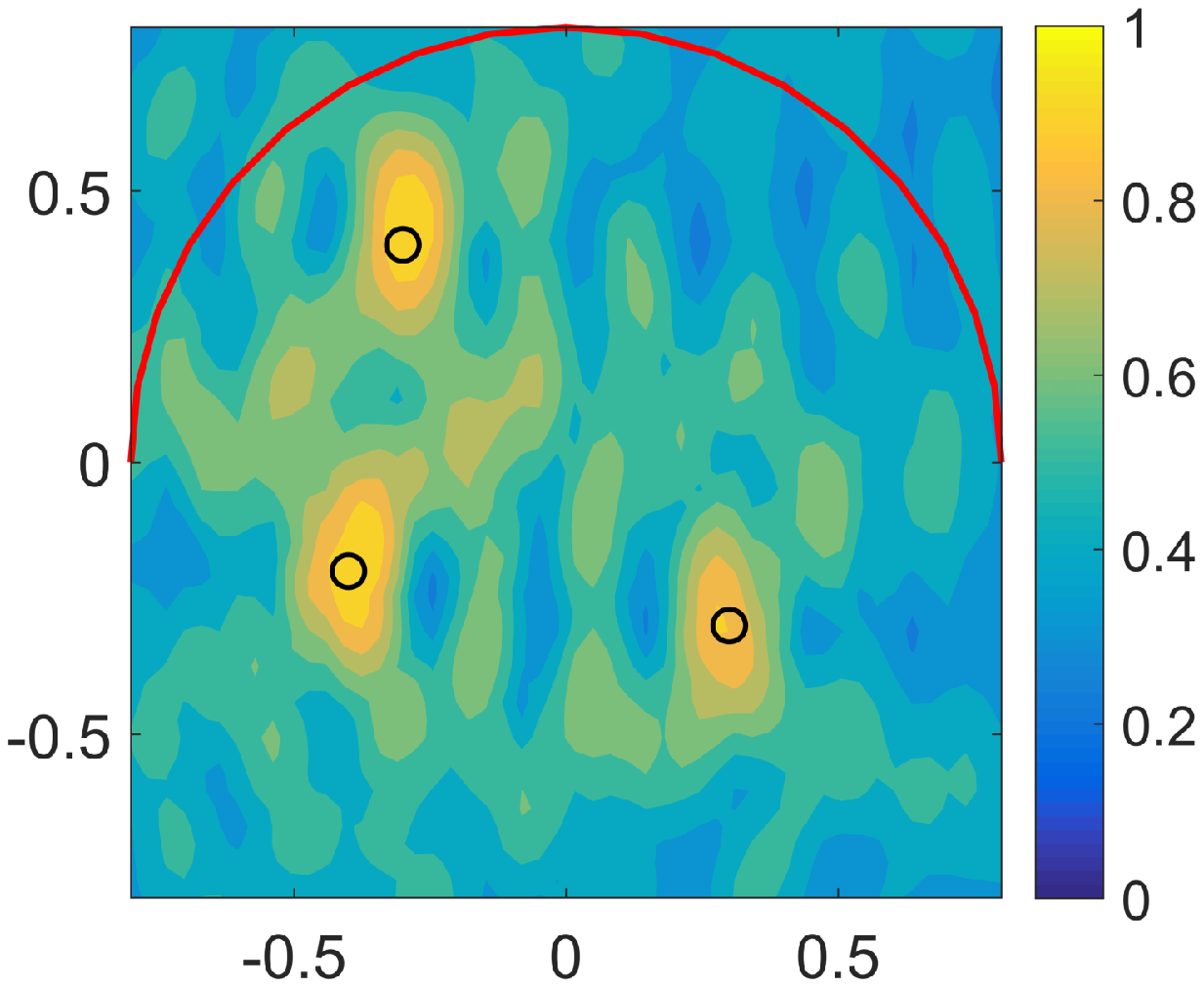}}\hspace*{0.01\textwidth}
	\subfloat[\label{LimitedViewEx1ClassicalMonoDSM}Map of $\mathcal{I}^{\mathrm{mono}}_{\mathrm{DSM}}(\fz)$]{\centering\includegraphics[width=.225\textwidth]{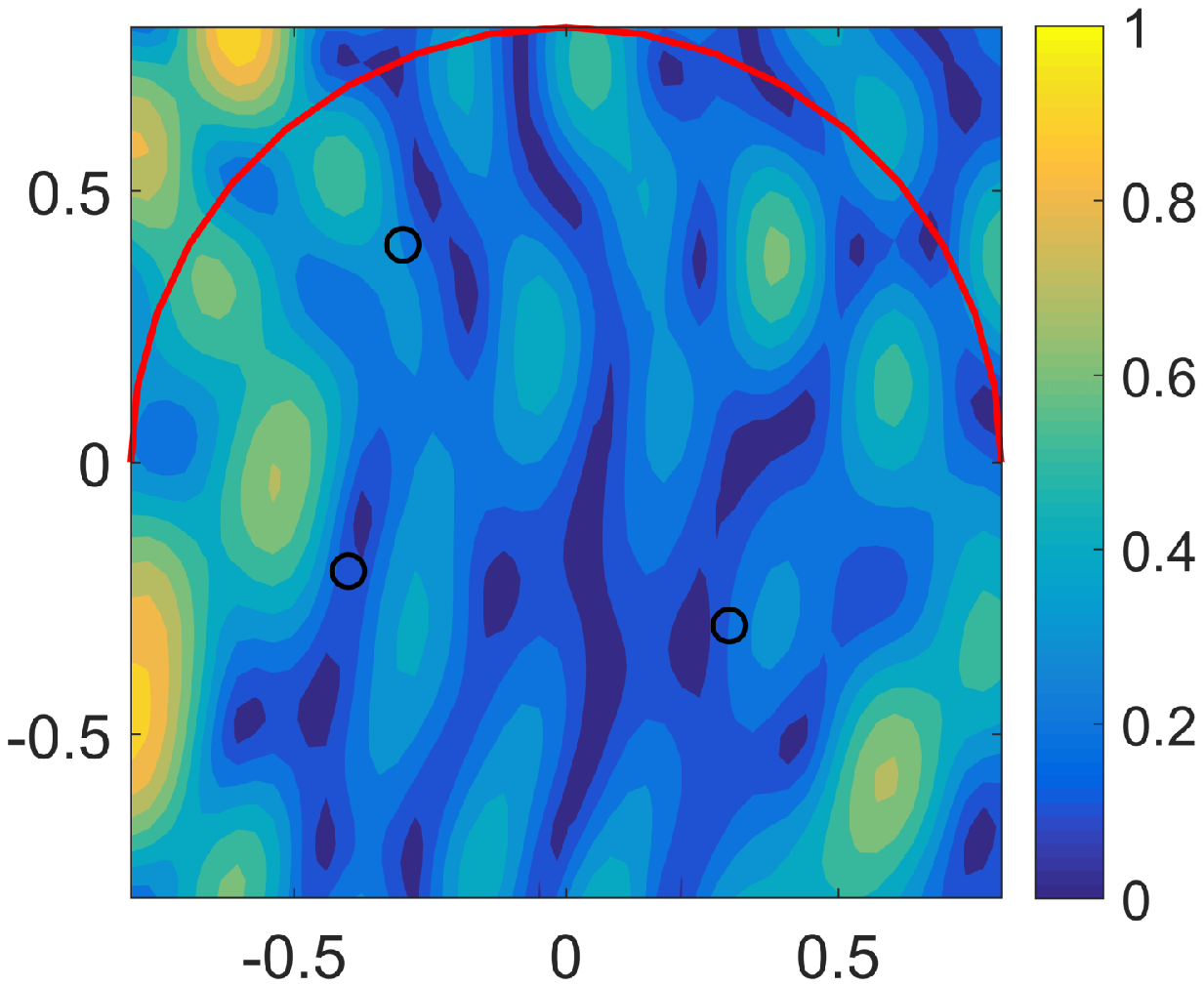}}\hspace*{0.01\textwidth}
	\subfloat[\label{LimitedViewEx1MonoMDSM}Map of $\mathcal{I}^{\mathrm{mono}}_{\mathrm{MDSM}}(\fz)$]{\centering\includegraphics[width=.225\textwidth]{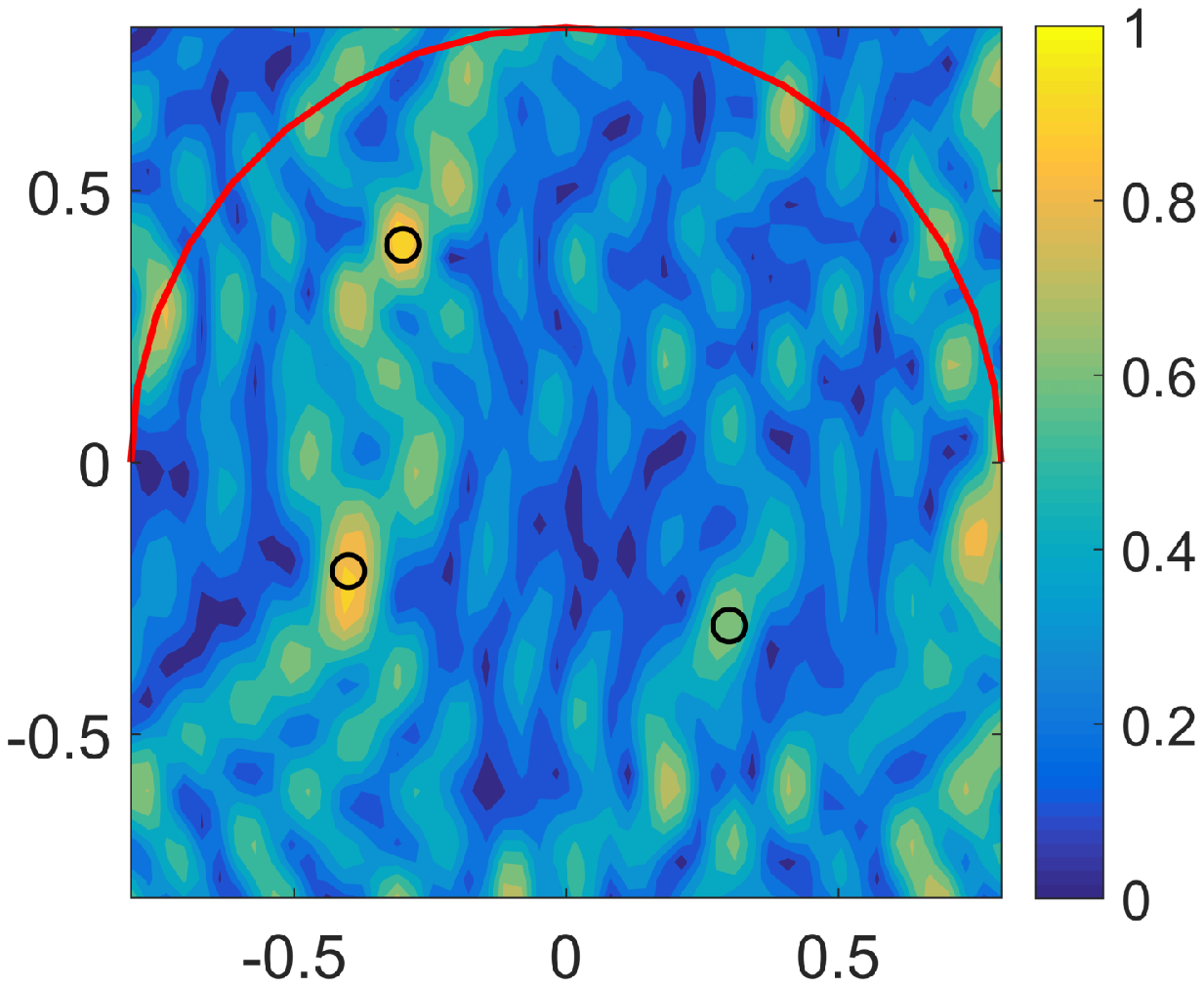}}\hspace*{0.01\textwidth}
	\subfloat[\label{LimitedViewEx1Jaccard}Jaccard index]{\includegraphics[width=.225\textwidth]{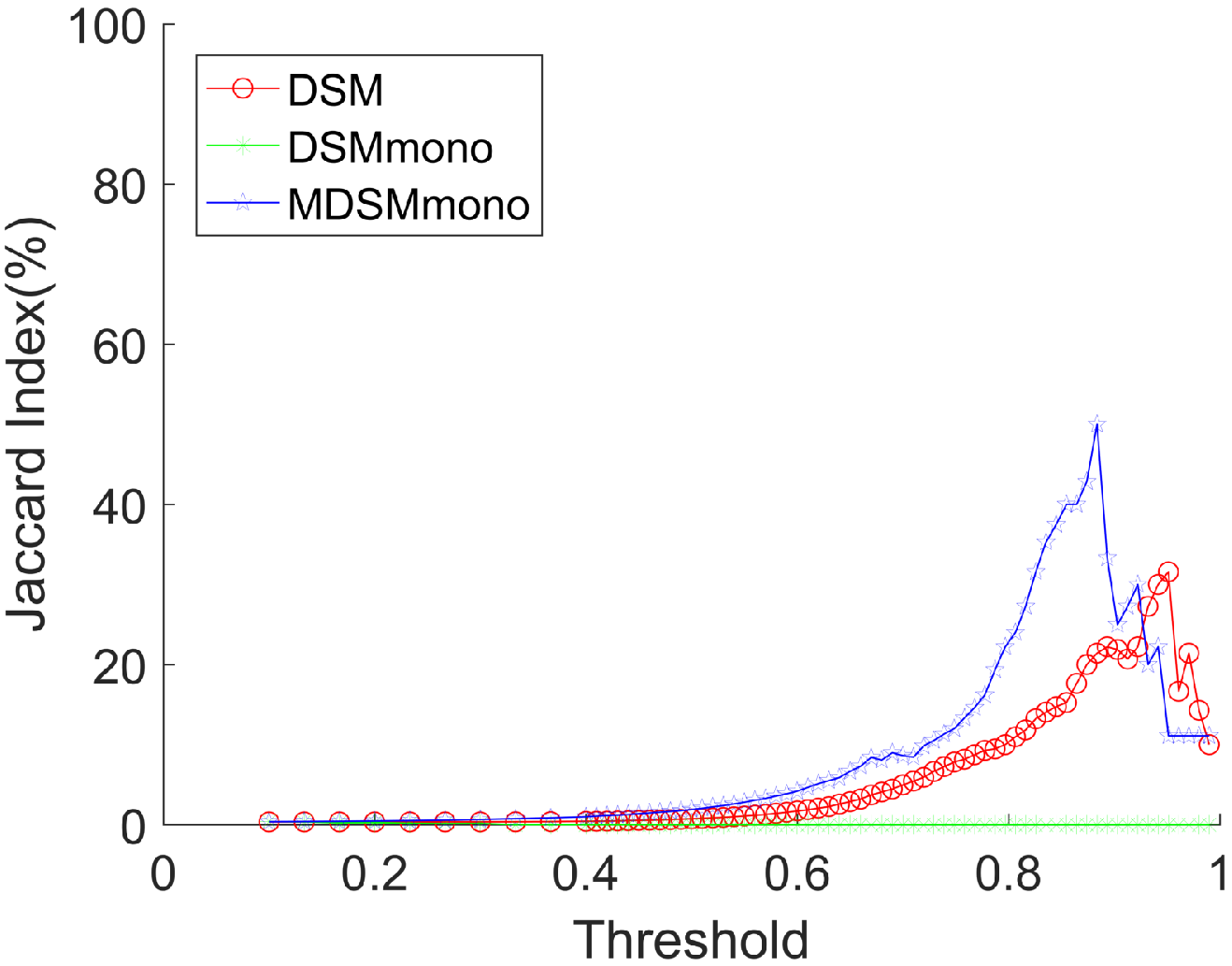}}
	\caption{\label{Result1-3}Simulation results in Example \ref{EX:3}. Red-colored solid-line describes the limited range of incident and observation directions.}
\end{figure*}

\section{Conclusion and perspective}\label{sec:6}
In this study, the application of DSM in the mono-static configuration for finding the location of small targets is considered in a 2D scalar configuration. Thanks to the use of the asymptotic expansion formula in the presence of small inhomogeneities and the far-field hypothesis, the mathematical structure of the indicator function of the traditional DSM is established and the reason for which it fails to image the defects is clearly identified. To overcome this miss-localization of the defects a modified DSM (MDSM) is proposed and its efficiency is theoretically shown. Numerical simulations are provided to support our theoretical results for various obstacles. 

Nevertheless, some improvements are still required as for, as an example, the near-field case for which the provided equations are no longer correct, while the multi-frequency version is also of interest and should be treated. 


%

\section*{Acknowledgment}

The authors would like to acknowledge D.~Lesselier for his valuable comments. W.-K.~Park was supported by the Basic Science Research Program of the National Research Foundation of Korea (NRF) funded by the Ministry of Education (No. NRF-2017R1D1A1A09000547).

\bibliographystyle{IEEEtran}
\bibliography{IEEEabrv,ReferencesURSI2}

%
%
%




\end{document}